\documentclass[reqno]{amsart}
\usepackage{amssymb,amsmath,amsthm,amsfonts,fancyhdr}
\usepackage{indentfirst}
\usepackage{mathrsfs}
\usepackage{setspace}
\usepackage{color}
\usepackage{cite}
\usepackage{bm}
\usepackage{lineno}
\usepackage{hyperref}
\usepackage{latexsym}
\usepackage{verbatim}
\usepackage{cases}
\usepackage{esint}
\theoremstyle{plain} 
\newtheorem{thm}{Theorem}[section]
\newtheorem*{thmA}{Theorem A}
\newtheorem*{thmB}{Theorem B}
\newtheorem{pro}[thm]{Proposition}
\newtheorem{lem}[thm]{Lemma}
\newtheorem{cor}[thm]{Corollary}
\theoremstyle{definition}
\newtheorem{defn}{Definition}
\theoremstyle{remark}
\newtheorem{rmk}[thm]{Remark}

\renewcommand{\det}{\mbox{det}}

  \numberwithin{equation}{section}
  \numberwithin{figure}{section}

\begin{document}

\title{\textbf{Liouville theorem for  fully nonlinear elliptic equations with the small  oscillation and the periodicity in $x$  and the periodic right hand term}}

\author{Lichun Liang}

\address{School of Mathematical
Sciences, Chongqing Normal University,
Chongqing, 401331, P.R. China}

\email{lianglichun@cqnu.edu.cn}

\begin{abstract}
In this paper, we study quadratic growth solutions $u$ of fully nonlinear elliptic equations of the form $F(D^2u,x)=f$ in $\mathbb{R}^n$, where $f$ is periodic and $F$ has the periodicity in $x$. Under the assumption that the oscillation of  $F(M,x)$ in $x$ is ``small",
we establish the existence and Liouville type results for quadratic  growth solutions, which can be expressed into the sum of a quadratic polynomial and a periodic function. Consequently, these results  are generalization of the existing results for  linear elliptic equations $a_{ij}D_{ij}u=0$ and   fully nonlinear elliptic equations $F(D^2u)=f$ with the periodic data.
\end{abstract}


\date{}
\maketitle

\section{Introduction}

\noindent

In this paper, we are concerned with quadratic growth solutions of fully nonlinear  elliptic equations of the form
\begin{equation}\label{eq1}
F(D^2u(x),x)=f(x),\ \ \ x\in \mathbb{R}^n,
\end{equation}
where $f$ is continuous and periodic, i.e.,
$$f(x+z)=f(x)$$
for all $x\in \mathbb{R}^n$ and $z\in \mathbb{Z}^n$.
$F(M,x)$ is a real valued  continuous  function defined on $\mathcal{S}^{n\times n}\times \mathbb{R}^n$, where $\mathcal{S}^{n\times n}$ is the space of all real $n\times n$ symmetric matrices. We assume that the operator  $F$ satisfies the following structure  conditions:

\textbf{(H1):}(Uniformly  ellipticity) There are  two  constants $0<\lambda \leq \Lambda < \infty$ such that
$$\lambda \|N\|\leq F(M+N,x)-F(M,x)\leq \Lambda\|N\|$$
for any $M, N\in \mathcal{S}^{n\times n}$ with $N\geq 0$ and $x\in \mathbb{R}^n$.


\textbf{(H2):} (Periodicity  in $x$) $$F(M,x+z)=F(M, x)$$
for all $x\in \mathbb{R}^n$, $z\in \mathbb{Z}^n$ and $M\in \mathcal{S}^{n\times n}$.

\textbf{(H3):} (Concavity in $M$)

$$\frac{1}{2}F(M,x)+\frac{1}{2}F(N,x) \leq F\left(\frac{1}{2}M+\frac{1}{2}N,x\right)$$
for all $x\in \mathbb{R}^n$ and $M, N\in \mathcal{S}^{n\times n}$.

In order to measure the oscillation of $F$ in the variable $x$, we consider the function
$$\beta(x,x_0)=\sup_{M\in \mathcal{S}^{n\times n}\setminus \{0\}}\frac{|F(M,x)-F(M,x_0)|}{\|M\|}.$$
Clearly, due to (H2), $\beta(x,x_0)$ is periodic in $x$ and $x_0$. We also introduce Pucci's extremal operators, i.e.,
$$\mathcal{M}^{+}(M)=\Lambda\sum_{\kappa_i>0} \kappa_i+\lambda \sum_{\kappa_i<0} \kappa_i\ \ \ \mbox{and}\ \ \ \mathcal{M}^{-}(M)=\lambda\sum_{\kappa_i>0} \kappa_i+\Lambda \sum_{\kappa_i<0} \kappa_i,$$
where $\kappa_i=\kappa_i(M) (i=1,\ldots,n)$ are the eigenvalues of $M\in \mathcal{S}^{n\times n}$.

In order to expound  the main motivations on this work, we would like to confine our attention to the development of Liouville type results for elliptic equations with the periodic data in the whole space.   These Liouville-type results are intimately connected with the study  of  minimal solutions of variational problems on a torus\cite{MS92,M86},  the dimension of the space of polynomial growth solutions of degree at most $d$\cite{L9,LW00,LW}, as well as the existence of correctors and the limiting equation in homogenization theory\cite{Ev92,C99}.
In  1989, by implementing  the tools from homogenization theory, Avellaneda and Lin \cite{AL}  first obtained a Liouville type result for  linear elliptic equations of divergence form
$\partial_i(a_{ij}(x)\partial_ju(x))=0$ in $\mathbb{R}^n$ with the periodic data.
Under the hypothesis that the coefficients $a_{ij}(x)$ are Lipschitz continuous and periodic, they showed that any polynomial growth solution of degree of at most $m$ must be a polynomial  with  periodic coefficients.  A few years later, Moser and Struwe \cite{MS92}considered quasilinear elliptic equations $-div (F_{p}(x,Du(x)))=0$ in $\mathbb{R}^n$, where $F(x,p)$ is periodic in $x$ and satisfies  convexity and suitable growth assumptions with respect to $p$. Using the Harnack inequality from the elliptic equation theory, they showed that any linear growth solution must be a linear function up to a periodic perturbation, which partially generalizes Avellaneda and Lin's results from the linear to the nonlinear case. Moreover, they also  achieved a simplified proof for the linear case without the Lipschitz continuous assumption on the coefficients.  For linear elliptic equations of non-divergence form $a_{ij}(x)D_{ij}u(x)=0$ in $\mathbb{R}^n$ with measurable and periodic coefficients, Li and Wang \cite{LW}  proved a result similar to that in \cite{AL}.  For Monge-Amp\`{e}re equations $\det (D^2u(x))=f(x)$ in $\mathbb{R}^n$ with $f$ being periodic, Li \cite{LL90} first established the existence result for entire convex quadratic polynomial growth solutions decomposed as a quadratic polynomial plus a periodic function. After ten years,  Caffarelli and Li \cite{CL04}  showed that any convex solution $u$ of $\det(D^2u(x))=f(x)$ in  $\mathbb{R}^n$ must be a quadratic polynomial  up to a periodic perturbation under the condition that $f$ is periodic and smooth. Recently, the smooth assumption on $f$ has been weakened into $f\in L^{\infty}(\mathbb{R}^n)$ by Li and Lu\cite{LL22}.
Considering fully nonlinear uniformlly elliptic equations $F(D^2u)=f$ in  $\mathbb{R}^n$ with the periodic hand term $f$, Li and Liang \cite{LL2024} established the same  Liouville type result as  Monge-Amp\`{e}re equations under the condition that the solution grows no faster than $|x|^2$ at infinity.

To clearly highlight the distinctive contributions of this paper, we would like to mention two important Liouville-type theorems.
One of them is related to  linear non-divergence  elliptic equations with periodic coefficients, while the other one deals with fully nonlinear uniformlly elliptic equations with the periodic hand term.
\bigskip
\begin{thmA}[Li-Wang \cite{LW}]
Let $a_{ij}$ be periodic and measurable function satisfying
\begin{equation*}
  \bar{\lambda} |\xi|^2\leq a_{ij}(x) \xi_i \xi_j\leq \bar{\Lambda} |\xi|^2
\end{equation*}
for any $x, \xi \in \mathbb{R}^n$, where $0<\bar{\lambda}\leq \bar{\Lambda} < \infty$. Assume that  $u\in W^{2,n}_{loc}(\mathbb{R}^n)$ is a strong solution of
$$a_{ij}D_{ij}u=0\ \ \mbox{in}\ \ \mathbb{R}^n$$
and satisfies
\begin{equation*}
   |u(x)|\leq C(1+|x|^2),\ \ \ x\in \mathbb{R}^n
 \end{equation*}
 for some constant $C>0$. Then there exist some $A\in \mathcal{S}^{n\times n}$, $b\in \mathbb{R}^n$ and periodic function $v\in C(\mathbb{R}^n)$ such that
 $$u(x)=\frac{1}{2}x^{T}Ax+b\cdot x+v(x).$$
 Moreover, if the coefficients $a_{ij}$ are  continuous, then there exists some positive definite  matrix $Q \in \mathcal{S}^{n\times n}$ such that
 $$tr (AQ)=0.$$
\end{thmA}
\bigskip

\begin{thmB}[Li-Liang \cite{LL2024}]
 Let $F\in C(\mathcal{S}^{n\times n})$   be   uniformly elliptic and  $f\in  C(\mathbb{R}^n)$ be periodic.
 Assume that $u$ is a viscosity solution of
 $$F(D^2u)=f\ \ \mbox{in}\ \ \mathbb{R}^n$$ and satisfies
 \begin{equation*}
   |u(x)|\leq C(1+|x|^2),\ \ \ x\in \mathbb{R}^n
 \end{equation*}
 for some constant $C>0$. Then, if either $n\geq 3$ and $F$ is concave (convex) or $n=2$,
 there exist some $A\in \mathcal{S}^{n\times n}$, $b\in \mathbb{R}^n$ and periodic function $v\in C(\mathbb{R}^n)$ such that $$u(x)=\frac{1}{2}x^{T}Ax+b\cdot x+v(x).$$
 Moreover, there exists some uniformly elliptic operator $\bar{F}\in C(\mathcal{S}^{n\times n})$ such that
 $$\bar{F}(A)=0.$$
 \end{thmB}

\begin{rmk}
For linear elliptic equations of non-divergence form $a_{ij}(x)D_{ij}u(x)=0$ in $\mathbb{R}^n$, there has been an interesting Liouville-type theorem established by Nirenberg \cite{LN54,LN56}.  This theorem states that if the coefficients $a_{ij}$ are close to some constant and $u=O(|x|^{2-\delta})$ at infinity for some $\delta>0$, then $u$ must be a linear function.
\end{rmk}
\bigskip

In this paper, our purpose is to investigate the  Liouville-type result for quadratic polynomial growth solutions $u$ of $F(D^2u,x)=f$ with the periodicity in $x$ and the periodic right hand term $f$. Under the assumption that the oscillation of  $F(M,x)$ in $x$ is ``samll", we can generalise Theorem A and B.

We introduce the space $\mathbb{T}$ consisting of all continuously periodic functions with zero mean, defined by
$$\mathbb{T}=\left\{v\in C(\mathbb{R}^n):v(x+z)=v(x)\ \mbox{for all}\ x\in \mathbb{R}^n\ \mbox{and}\ z\in \mathbb{Z}^n, \fint_{[0,1]^n}v\,dx=0\right\}.$$

We can now formulate our main results.
\bigskip

\noindent
\textbf{1.1. Uniformly elliptic equations in the whole space.} The following theorem is the cornerstone of the existence result for quadratic  growth solutions of (\ref{eq1}).
\bigskip
\begin{thm}\label{th3}
Let $F\in C^2(\mathcal{S}^{n\times n}\times \mathbb{R}^n)$   satisfy (H1), (H2) and (H3) and $f\in C^{\alpha}(\mathbb{R}^n)$ be periodic for some $\alpha\in (0,1)$. Assume that there exists a constant $\bar{C}>0$ such that
\begin{equation}\label{eq43}
   \left(\fint_{Q_r(x_0)} |\beta(x,x_0)|^n\,dx\right)^{\frac{1}{n}}\leq \bar{C} r^\alpha
 \end{equation}
for all $r\leq 2$ and $x_0\in Q_1$.
Then for any $A\in \mathcal{S}^{n\times n}$,
\begin{equation}\label{eq2}
F(A+D^2v(x),x)-\fint_{[0,1]^n}F(A+D^2v(x),x)\,dx=f(x)- \fint_{[0,1]^n}f(x)\,dx
\end{equation}
has a unique  solution $v\in C^{2}(\mathbb{R}^n)\cap \mathbb{T}$.
\end{thm}
\bigskip
As a matter of fact,  we can dispense with the smooth assumption on $F$ and $f$ via a smooth approximation of $F$ and $f$ and the H\"{o}lder estimate.
\bigskip
 \begin{cor}\label{th4}
 Let $F\in C(\mathcal{S}^{n\times n}\times \mathbb{R}^n)$   satisfy (H1), (H2) and (H3) and $f\in C(\mathbb{R}^n)$ be periodic. Assume that  there exists a positive  constant $\beta_0$ depending only on $n$, $\lambda$, $\Lambda$ and $p>n$ such that
 \begin{equation}\label{eq46}
   \left(\fint_{Q_r(x_0)} |\beta(x,x_0)|^n\,dx\right)^{\frac{1}{n}}\leq \beta_0
 \end{equation}
for all $r\leq 2$ and $x_0\in Q_1$. Then for any $A\in \mathcal{S}^{n\times n}$, there exists a unique $\alpha\in \mathbb{R}$ such that
\begin{equation*}
  F(A+D^2v,x)-\alpha=f(x)-\fint_{[0,1]^n}f(x)\,dx
\end{equation*}
has a unique viscosity solution  $v\in \mathbb{T}$.
 \end{cor}
 \bigskip
\begin{rmk}\label{rm1}
The constant $\beta_0$ coming from \cite[Theorem 7.1]{CC95} ensures that the interior   $W^{2, p}$ estimate holds for viscosity solutions of $F(D^2u,x)=f(x)$. More importantly, taking into consideration the periodicity of $\beta(x,x_0)$ and the condition (\ref{eq46}), we can establish the adimensional   $W^{2,p}$ estimate, which will be useful in the establishment of the following Liouville-type result.
 \end{rmk}

We introduce the homogenisation operator $\bar{F}$ as follows.

\begin{defn}
Let $F\in C(\mathcal{S}^{n\times n}\times \mathbb{R}^n)$   be elliptic and $f\in C(\mathbb{R}^n)$ be periodic. For any $A\in \mathcal{S}^{n\times n}$, if  there exists a unique $\alpha\in \mathbb{R}$ such that
\begin{equation*}
  F(A+D^2v,x)-\alpha=f(x)-\fint_{[0,1]^n}f(x)\,dx
\end{equation*}
has a unique viscosity solution  $v\in \mathbb{T}$, then  the homogenisation operator $\bar{F}: \mathcal{S}^{n\times n} \rightarrow \mathbb{R}$ is defined by
$$A\mapsto \bar{F}(A)=\alpha.$$
\end{defn}
\bigskip
\begin{rmk}
\emph{(i).} Here we emphasize the dependence of  $\bar F$ on $A$ and actually, it may also depends on $f$.

\emph{(ii).}  If $F$ and $f$ satisfy the hypotheses of Corollary \ref{th4}, then we can define the homogenisation operator $\bar{F}$.  In particular, if $F(M,x)=a_{ij}(x)M_{ij}$, then we obtain the explicit homogenization operator
 \begin{align*}
   \bar{F}(A):= & \fint_{Q_1}a_{ij}(x)(A_{ij}+D_{ij}v(x))\,dx \\
   =& A_{ij}\fint_{Q_1}a_{ij}(x)m(x)\,dx-\fint_{Q_1}f(x)m(x)\,dx+\fint_{Q_1}f(x)\,dx,
 \end{align*}
 where $m$ is a periodic function and the unique solution of the problem $D_{ij}(a_{ij}(x)m(x))=0$ in $\mathbb{R}^n$  with $\fint_{Q_1}m(x)\,dx=1$ (see \cite[Theorem 2]{AL89} and \cite{BLP78}).  Clearly, the positive definite  matrix $Q$ in Theorem A is actually $\fint_{Q_1}a_{ij}(x)m(x)\,dx$.

 \emph{(iii).} If the operator $F$ is the $k$-Hessian operator, then the homogenization operator $\bar{F}$ is the $k$-Hessian operator itself (see \cite[Theorem 1.14]{LL2024}).
 \end{rmk}

 \bigskip
Consequently, we will give the existence of quadratic polynomial growth solutions of (\ref{eq1}).
\bigskip
\begin{thm}\label{th1}
Let $F\in C(\mathcal{S}^{n\times n}\times \mathbb{R}^n)$   satisfy (H1), (H2) and (H3) and  $f\in  C(\mathbb{R}^n)$ be periodic. Assume that  the oscillation $\beta(x,x_0)$ satisfies the condition (\ref{eq46}).
Then for any $A\in \mathcal{S}^{n\times n}$, $b\in \mathbb{R}^n$  and $c\in \mathbb{R}$, (\ref{eq1}) has a viscosity solution $u$ satisfying $u(x)=\frac{1}{2}x^{T}Ax+b\cdot x+c+v(x)$ for some  $v\in \mathbb{T}$ if and only if $\bar{F}(A)=\fint_{[0,1]^n}f(x)\,dx$. In particular, $v$ is uniquely determined by $A$.
\end{thm}
\bigskip
The following  Liouville theorem indicates that any quadratic polynomial growth solution of (\ref{eq1})  must be a quadratic polynomial plus a periodic function.
\bigskip
 \begin{thm}\label{th2}
 Let $F\in C(\mathcal{S}^{n\times n}\times \mathbb{R}^n)$   satisfy (H1), (H2) and (H3) and  $f\in  C(\mathbb{R}^n)$ be periodic.
  Assume that $u$ is a viscosity solution of (\ref{eq1}) and satisfies
 \begin{equation}\label{eq36}
   |u(x)|\leq C(1+|x|^2),\ \ \ x\in \mathbb{R}^n
 \end{equation}
 for some constant $C>0$. Then, if  the oscillation $\beta(x,x_0)$ satisfies the condition (\ref{eq46}), there exist some $A\in \mathcal{S}^{n\times n}$ with $\bar{F}(A)=\fint_{[0,1]^n}f(x)\,dx$, $b\in \mathbb{R}^n$, $c\in \mathbb{R}$ and $v\in \mathbb{T}$ such that $$u(x)=\frac{1}{2}x^{T}Ax+b\cdot x+c+v(x).$$
 \end{thm}
\bigskip
\begin{rmk}
\emph{(i).} We consider a simple case that the oscillation $\beta(x,x_0)$ is continuous at $x_0$. The condition (\ref{eq46}) is superfluous in Theorem \ref{th1}, whereas the proof of Theorem \ref{th2} requires it to obtain the $W^{2,p}$ estimate for  the blow-down sequence for solutions.

\emph{(ii).}  If the oscillation of $\beta(x,x_0)$ in $Q_1$ is small enough, then it is clear that the condition (\ref{eq46}) holds.

\emph{(iii).}  If $n=2$ or $F(M,x)=a_{ij}(x)M_{ij}$, we can apply the difference argument previously used by Li and Liang \cite{LL2024} to obtain  the Liouville type result. Moreover, for linear elliptic equations, the small  oscillation condition (\ref{eq46}) can be dropped out. Actually, by the H\"{o}lder estimate, we see that there exists some constant $C>0$ such that
\begin{equation*}
  \begin{split}
 r^{\alpha}[u]_{\alpha; Q_{2r}}& \leq C (\|u\|_{L^{\infty}(Q_{4r})}+16r^2\|\tilde{f}\|_{L^{\infty}(Q_{4r})})\\
      & \leq C (1+16r^2+16r^2\|f\|_{L^{\infty}(\mathbb{R}^n)}).
  \end{split}
\end{equation*}
Therefore, for any $x\in Q_r$ and orthogonal basis $e_k$ $(k=1,\ldots,n)$, this implies that
$$|u(x+e_k)-u(x)|\leq C(1+|x|^{2-\alpha}).$$
From Theorem A, it follows that the first order difference $u(x+e_k)-u(x)$ must be a affine function, which immediately yields that $u$ is a second order polynomial plus a periodic function.
\end{rmk}

\bigskip
\noindent
\textbf{1.2. Uniformly elliptic equations on exterior domains.} Theorem \ref{th1} enables us to establish the following existence theorem for the Dirichlet problem on exterior domains with prescribed asymptotic behavior at infinity.
\bigskip
\begin{thm}\label{th10}
Let $F\in C(\mathcal{S}^{n\times n}\times \mathbb{R}^n)$   satisfy (H1), (H2) and (H3) and $f\in C^\alpha(\mathbb{R}^n)$  for some $0<\alpha<1$ be periodic. Assume that  the oscillation $\beta(x,x_0)$ satisfies the condition (\ref{eq46}), that $A\in \mathcal{S}^{n\times n}$ satisfies $\bar{F}(A)=\fint_{[0,1]^n}f(x)\,dx$ and that $\Omega\subset \mathbb{R}^n$ is a domain satisfying a uniform interior sphere condition. Then for any $b\in \mathbb{R}^n$ and $\varphi\in C(\partial \Omega)$, there exist some  $c\in \mathbb{R}$ and $v\in \mathbb{T}$ such that
there exists a  viscosity solution $u\in C(\mathbb{R}^n \backslash \overline{\Omega})$  of
\begin{equation*}
   \left\{
\begin{aligned}
   &F(D^2u,x)=f\ \ \mbox{in}\ \ \mathbb{R}^n \backslash \overline{\Omega}\\
&u=\varphi\ \  \mbox{on}\ \ \partial \Omega
\end{aligned}
\right.
 \end{equation*}
satisfying
$$\lim_{|x|\rightarrow \infty}\left(u(x)-\frac{1}{2}x^{T}Ax-b\cdot x-c-v(x)\right)=0,$$
where $v\in \mathbb{T}$ is a unique viscosity solution of  $F(A+D^2v,x)=f$ in $\mathbb{R}^n$. Furthermore, if   $\frac{\Lambda}{\lambda}< n-1$, we have the following estimate
\begin{equation}\label{eq39}
  \left|u(x)-\frac{1}{2}x^{T}Ax-b\cdot x-c-v(x)\right|\leq C|x|^{1-(n-1)\frac{\lambda}{\Lambda}}, \ \ \ x\in \mathbb{R}^n \backslash \overline{\Omega},
\end{equation}
where $C$ is a positive constant.
\end{thm}
\bigskip
\begin{rmk}
\emph{(i).} The domain $\Omega$ is said to satisfy  a uniform interior sphere condition if there is a constant $\rho>0$ such that for any $x_0\in \partial \Omega$ there exists a ball $B_{\rho}(z_{x_0})\subset \Omega$ with $\overline{B_{\rho}(z_{x_0})}\cap \partial \Omega={x_0}$ for some $z_{x_0}\in \Omega$.

\emph{(ii).} Here $|x|^{1-(n-1)\frac{\lambda}{\Lambda}}$ is the fundamental solution of the Pucci's operators (see \cite{L01,ASS11} for more details).
\end{rmk}
\bigskip
As an extension of Liouville theorem (Theorem \ref{th2}), we investigate the asymptotic behavior at infinity of quadratic  growth viscosity solutions  of  (\ref{eq1}) in  exterior domains.
 \bigskip
 \begin{thm}\label{th9}
 Let $F\in C(\mathcal{S}^{n\times n}\times \mathbb{R}^n)$   satisfy (H1), (H2) and (H3) with (\ref{eq43}) and 
 and $f\in C^\alpha(\mathbb{R}^n)$  for some $0<\alpha<1$ be periodic.
 Assume that  the oscillation $\beta(x,x_0)$ satisfies the condition (\ref{eq46}), that $\beta(x,x_0)$ is H\"{o}lder continuous in $Q_1$ and that
 $u$ is a viscosity solution of
 $$F(D^2u,x)=f\ \ \ \mbox{in}\ \ \ \mathbb{R}^n \backslash \overline{B_1}$$
and satisfies
 \begin{equation*}
   |u(x)|\leq C|x|^2,\ \ \ x\in \mathbb{R}^n \backslash \overline{B_1}
 \end{equation*}
 for some constant $C>0$. Then, if   $\frac{\Lambda}{\lambda}< n-1$, there exist some $A\in \mathcal{S}^{n\times n}$ with $\bar{F}(A)=\fint_{[0,1]^n}f(x)\,dx$, $b\in \mathbb{R}^n$, $c\in \mathbb{R}$ and $v\in \mathbb{T}$ such that
 $$\left|u(x)-\frac{1}{2}x^{T}Ax-b\cdot x-c-v(x)\right| \leq C|x|^{1-(n-1)\frac{\lambda}{\Lambda}},\ \ \ x\in \mathbb{R}^n \backslash \overline{B_1}$$
for some cosnstant $C>0$, where $v\in \mathbb{T}$ is a unique viscosity solution of  $F(A+D^2v,x)=f$ in $\mathbb{R}^n$.
 \end{thm}
\bigskip
\begin{rmk}
For fully nonlinear uniformly elliptic equations $F(D^2u)=f$ in exterior domains, there have been substantial asymptotic results.
When the right hand $f$ is a constant, Li, Li and Yuan \cite{LLY20} obtained the asymptotic behavior at infinity under the smooth assumption on $F$ for the high dimensional case, whereas the two dimensional case was recently solved by Li and Liu \cite{LL24}. Subsequently, Lian and Zhang \cite{LZP} obtained an asymptotic  result  without the smooth assumption on $F$. When the right hand $f$ is a periodic function, Li and Liang \cite{LL2024} established an asymptotic result.
\end{rmk}
\bigskip

This paper is organized as follows. In Section 2, we use the method of continuity to prove Theorem \ref{th3} and  obtain Corollary \ref{th4} by a smooth approximation. Furthermore, with the help of the uniform ellipticity  for  the homogenisation operator, we  achieve the proof  of Theorem \ref{th1}. By considering the blow-down sequence for solutions, we determine the second order term, thereby  reducing the proof of Theorem \ref{th2} to analogous analysis to the Monge-Amp\`{e}re equation case \cite{CL04}. In Section 3, we use Theorems \ref{th1} and \ref{th2} to prove
Theorems \ref{th10}  and \ref{th9}.
\bigskip

We close this section by introducing some notations.

$\bullet$ For $r>0$, $Q_r=[-\frac{r}{2},\frac{r}{2}]^n$.

$\bullet$ For $x\in \mathbb{R}^n$ and $r>0$, $B_r(x)=\{y\in \mathbb{R}^n: |x-y|<r\}$ and $B_r=B_r(0)$.

$\bullet$ For $M\in \mathcal{S}^{n\times n}$ and $r>0$, $B^r(x)=\{N\in \mathbb{R}^n: |M-N|<r\}$ and $B^r=B^r(0)$.

$\bullet$ $e_1=(1,0,\ldots, 0),\ldots, e_n=(0,0,\ldots, 1)$.

$\bullet$   Set $$E=\{k_1 e_1+\cdots +k_n e_n: k_1,\ldots, k_n\ \ \mbox{are integers}\},$$
and the second order difference quotients
$$\Delta_{h}^2u(x)=\frac{u(x+h)+u(x-h)-2u(x)}{\|h\|^2}\ \ \mbox{for}\ \ h\in \mathbb{R}^n.$$

$\bullet$  The  H\"{o}lder seminorm
$$[u]_{\alpha; \Omega}=\sup_{x,y \in \Omega}\frac{|u(x)-u(y)|}{|x-y|^\alpha}.$$
\bigskip
\section{Uniformly elliptic equations in the whole space}
\noindent

Throughout this section, we set $X=C^{2,\alpha}(\mathbb{R}^n) \cap \mathbb{T}$ and $Y=C^{ \alpha}(\mathbb{R}^n) \cap \mathbb{T}$ for some $0<\alpha<1$.
In addition, we  assume $F(0,x)=0$ for all $x\in \mathbb{R}^n$.

\subsection{Existence}
\noindent

Firstly, we will prove Theorem \ref{th3} by using the method of continuity. Secondly, we will complete the proof of Corollary \ref{th4} by a smooth approximation of $F$ and $f$ and the H\"{o}lder estimate.

We need the following:
\begin{pro}\label{pr1}
Let $a_{ij}$ be periodic and of class $C^{ \alpha} (0<\alpha<1)$ satisfying
\begin{equation}\label{eq23}
  \bar{\lambda} |\xi|^2\leq a_{ij}(x) \xi_i \xi_j\leq \bar{\Lambda} |\xi|^2
\end{equation}
for any $x, \xi \in \mathbb{R}^n$, where $0<\bar{\lambda}\leq \bar{\Lambda} < \infty$.  Then for any periodic function $f\in C^{\alpha}(\mathbb{R}^n)$, the problem
$$a_{ij}D_{ij}v-\fint_{Q_1}a_{ij}D_{ij}v\,dx=f-\fint_{Q_1}f\,dx$$
has a unique solution $v\in C^{2,\alpha}(\mathbb{R}^n)$ being periodic.
\end{pro}
\begin{proof}
Using  the method of continuity, we can complete the proof. A detailed proof of this theorem was provided by Li and Liang \cite{LL2024}.
\end{proof}

\bigskip
\begin{proof}[Proof of Theorem \ref{th3}]
\emph{Uniqueness.} We let $v, w\in C^{2}(\mathbb{R}^n)\cap \mathbb{T}$ be  solutions of (\ref{eq2}). It follows that
either
$$F(A+D^2v(x),x)\leq F(A+D^2w(x),x)$$
or
$$F(A+D^2w(x),x)\leq F(A+D^2v(x),x).$$
Clearly, by the strong maximum principle and $v, w\in \mathbb{T}$, we obtain  $v=w$.

\emph{Existence.}
The proof mainly relies  on the method of continuity. Without loss of generality we  assume $A=0$, $\fint_{Q_1}f(x)\,dx=0$ and $f\in C^{\alpha}(\mathbb{R}^n)$ for some small enough $\alpha>0$.

Now we  consider the  map $ \mathcal{F} :   X \times [0,1]\longrightarrow Y$ defined by
 $$(v,t)\longmapsto \mathcal{F}(v,t)= F(D^2v,x)-\fint_{Q_1}F(D^2v,x)\,dx-tf,$$
and the set
$$\mathcal{T}:=\{t\in [0,1]: \mathcal{F}(v_t,t)=0 \ \ \mbox{for some}\ \ v_t\in X\}.$$
It is obvious that $0\in \mathcal{T}$ since $v_0=0$  is a unique solution. So to establish the existence of solutions $v\in X$  of (\ref{eq2}), i.e., $1\in \mathcal{T}$, it suffices to show that $\mathcal{T}$ is both open and closed in $[0,1]$.

\emph{Step 1}: \textbf{$\mathcal{T}$ is closed.} Let $v_t\in X$ denote a solution of $ \mathcal{F}(v_t,t)=0$.
 From the Evans-Krylov theory (see \cite[Theorem 8.1]{CC95}), we have
$$ \|v_t\|_{C^{2,\alpha} (Q_1)} \leq C \left(\|v_t\|_{L^{\infty}(Q_2)}+\|f\|_{C^{\alpha}(Q_2)}+\left|\fint_{Q_1}F(D^2v_t,x)\,dx\right|\right),$$
where  $C>0$  depends only on $n$, $\lambda$, $\Lambda$, $\alpha$ and $\bar{C}$.
By the aid of the uniform ellipticity condition for $F$ and the interpolation inequality,
we arrive at the estimate
\begin{equation*}
  \|v_t\|_{C^{2,\alpha} (Q_1)}\leq C (\|v_t\|_{L^{\infty}(Q_1)}+\|f\|_{C^{\alpha}(Q_1)}).
\end{equation*}

Next we need to show that
there exists a constant $C>0$ depending only on $\lambda$, $\Lambda$ and  $n$ such that
\begin{equation}\label{eq4}
  \|v_t\|_{L^{\infty}(Q_1)}\leq C \|f\|_{L^{\infty}(Q_1)}.
\end{equation}
To see this, since $v_t$ has a local minimum at $z\in Q_1$ and a local maximum  at $y\in Q_1$, by the uniform ellipticity for $F$, we obtain
$$    -tf(z)\leq \fint_{Q_1}F(D^2v_t,x)\,dx \leq  -tf(y),$$
which implies that
$$\left|\fint_{Q_1}F(D^2v_t,x)\,dx\right|\leq \|tf\|_{L^{\infty}(Q_1)}\leq\|f\|_{L^{\infty}(Q_1)}.$$
Combining this with the  Harnack inequality, we have
$$\|v_t-\min_{Q_1}v_t\|_{L^{\infty}(Q_1)}\leq C\|f\|_{L^{\infty}(Q_1)},$$
where $C>0$ depends only on $\lambda$, $\Lambda$ and  $n$. Consequently, we establish the estimate (\ref{eq4}), since $v_t\in X$.
Finally,  we obtain a priori estimates
\begin{equation*}
  \|v_t\|_{C^{2,\alpha} (Q_1)}\leq C \|f\|_{C^{\alpha}(Q_1)}
\end{equation*}
for some constant $C>0$ depending only on $n$, $\lambda$, $\Lambda$, $\alpha$ and $\bar{C}$, from which the closeness of $\mathcal{T}$ follows clearly.

 \emph{Step 2}: \textbf{$\mathcal{T}$ is open.} Openness follows from the implicit function theorem in Banach spaces. It is easy to check that $\mathcal{F}$ is of class $C^1$ and the Frech\`{e}t differential of $\mathcal{F}$ at $(v,t)$ with respect to $v$ 
 is given by
 $$h\mapsto D_v\mathcal{F}(v,t)[h]=F_{ij}(D^2v,x)D_{ij}h-\fint_{Q_1}F_{ij}(D^2v,x)D_{ij}h\,dx.$$
From Proposition \ref{pr1}, we see easily that $D_v\mathcal{F}(v,t)$ is a linear isomorphism between $X$ and $Y$.
 \end{proof}
\bigskip
By the aid of  Theorem \ref{th3}, we prove Corollary \ref{th4}.
\bigskip
\begin{proof}[Proof of Corollary \ref{th4}]
Without loss of generality we assume $A=0$ and $\fint_{[0,1]^n}f(x)\,dx=0$. We consider two nonnegative functions $\varphi\in C^{\infty}(\mathcal{S}^{n\times n})$ and $\eta\in C^{\infty}( \mathbb{R}^n)$ such that
$$\mbox{supp}\ \varphi \subseteq \overline{B^1},\ \ \ \mbox{supp}\ \eta \subseteq  \overline{B_1},\ \ \ \int_{\mathcal{S}^{n\times n}} \varphi\,dx=1\ \ \ \mbox{and}\ \ \ \int_{ \mathbb{R}^n} \eta \,dx=1,$$
where $B_1$ is the unit ball in $\mathbb{R}^n$ and $B^1$ is the unit ball in $\mathcal{S}^{n\times n}$.
For any $\epsilon >0$, set

$$\varphi_\epsilon(M):=\frac{1}{\epsilon^{n^2}}\varphi\left(\frac{M}{\epsilon}\right),\ \ \ M\in \mathcal{S}^{n\times n} \ \ \ \mbox{and}\ \ \ \eta_\epsilon(x):= \frac{1}{\epsilon^{n}}\eta\left(\frac{x}{\epsilon}\right),\ \ \ x\in \mathbb{R}^n.$$
Let $F_\epsilon$ be the mollification of $F$ in $\mathcal{S}^{n\times n}\times \mathbb{R}^n$ and $f_\epsilon$ be the mollification of $f$ in $\mathbb{R}^n$.  Then we define the oscillation of $F_\epsilon$ in the variable $x$, that is,
$$\widetilde{\beta_\epsilon}(x,x_0)=\sup_{M\in \mathcal{S}^{n\times n}}\frac{|F_\epsilon(M,x)-F_\epsilon(M,x_0)|}{1+\|M\|}.$$
Recalling that
\begin{equation*}
  \begin{split}
      F_\epsilon(M,x)-F_\epsilon(M,x_0)& =\int_{\mathcal{S}^{n\times n}\times \mathbb{R}^n} F(z,N)\varphi_\epsilon(M-N)\widetilde{\eta_\epsilon}(z)\,dzdN\\
      &=\int_{\{B_\epsilon(x)\cup B_\epsilon(x_0) \}\times B^\epsilon(M)} F(z,N)\varphi_\epsilon(M-N)\widetilde{\eta_\epsilon}(z)\,dzdN,
          \end{split}
\end{equation*}
where $\widetilde{\eta_\epsilon}(z)=\eta_\epsilon(x-z)-\eta_\epsilon(x_0-z)$,
we can show that
\begin{align*}
    \frac{\widetilde{\beta_\epsilon}(x,x_0)}{|x-x_0|}&=\sup_{M\in \mathcal{S}^{n\times n}}\frac{|F_\epsilon(M,x)-F_\epsilon(M,x_0)|}{(1+\|M\|)|x-x_0|} \\
  &\leq \sup_{M\in \mathcal{S}^{n\times n}} \int_{\mathcal{S}^{n\times n}\times \mathbb{R}^n} \frac{|F(z,N)|}{1+\|M\|}\varphi_\epsilon(M-N)\frac{\widetilde{\eta_\epsilon}(z)}{|x-x_0|}\,dzdN\\
  & \leq \sup_{M\in \mathcal{S}^{n\times n}} \int_{\mathcal{S}^{n\times n}\times \mathbb{R}^n} \beta(z,0)\frac{\|N\|}{1+\|M\|}\varphi_\epsilon(M-N)\sup_{\mathbb{R}^n}|D\eta_\epsilon|\,dzdN\\
  & \leq \sup_{M\in \mathcal{S}^{n\times n}} \int_{\{B_\epsilon(x)\cup B_\epsilon(x_0) \}} \beta(z,0) |D\eta_\epsilon|\,dz \int_{B^\epsilon(M) }\frac{\|N\|}{1+\|M\|}\varphi_\epsilon(M-N)\,dN\\
  & < +\infty,
\end{align*}
which implies that $\widetilde{\beta_\epsilon}(x,x_0)$ is Lipschitz continuous.

From  Theorem \ref{th3},
it follows that there is a unique solution $v_\epsilon\in Y$ such that
\begin{equation}\label{eq6}
  F_\epsilon(D^2v_\epsilon,x)-\fint_{Q_1}F_\epsilon(D^2v_\varepsilon,x)\,dx=f_\epsilon.
\end{equation}
Clearly, by $F(0,x)=0$, we note  that  $|F_\epsilon(0,x)|\leq 1$ for small enough $\epsilon>0$. Furthermore, since $v_\epsilon(y)=\max_{Q_1}v_\epsilon$ is a subsolution and $v_\epsilon(z)=\min_{Q_1}v_\epsilon$ is a supersolution,
 we  know that
$$-1-\sup_{Q_1}f_\epsilon \leq   F_\epsilon(0,z)-f_\epsilon(z)\leq \fint_{Q_1}F_\epsilon(D^2v_\epsilon,x)\,dx \leq  F_\epsilon(0,y)-f_\epsilon(y)\leq 1
-\inf_{Q_1}f_\epsilon,$$
that is,
$$\left|\alpha_\epsilon:=\fint_{Q_1}F_\epsilon(D^2v_\epsilon,x)\,dx\right|\leq 1+\|f_\epsilon\|_{L^{\infty}(Q_1)}\leq 1+\|f\|_{L^{\infty}(Q_1)}.$$
Combining this  with  the Harnack inequality and the interior H\"{o}lder estimate, we have
$$\|\bar{v}_\epsilon:=v_\epsilon-\min_{Q_1}v_\epsilon\|_{C^\alpha(Q_1)}\leq C(1+\|f\|_{L^{\infty}(Q_1)}),$$
where $0<\alpha <1$ and $C>0$ depends only on $\lambda$, $\Lambda$ and  $n$. Then we conclude that there exists a periodic function $\bar{v}$ and a real number $\alpha$ such that, up to a subsequence,
$$\bar{v}_\epsilon \rightarrow \bar{v}\ \ \ \mbox{in}\ \ \ C(Q_1)\ \ \ \mbox{and}\ \ \ \ \alpha_\epsilon \rightarrow \alpha \ \ \mbox{as}\ \ \epsilon\rightarrow 0.$$
Hence, letting $\epsilon\rightarrow 0$ in (\ref{eq6}), we obtain
\begin{equation}\label{eq7}
  F(D^2v,x)-\alpha=f \ \ \ \mbox{in}\ \ \ \mathbb{R}^n
\end{equation}
in the viscosity sense, where $v=\bar{v}-\fint_{Q_1}\bar{v}\,dx.$

It remains to show the uniqueness of $\alpha$. Indeed, suppose that $w\in \mathbb{T}$ and $\tilde{\alpha}$ satisfy (\ref{eq7}).
In  view of \cite[Theorem 7.1]{CC95},  we have $v, w\in W^{2,p}$ for $p>n$. Therefore, from the strong maximum principle for strong solutions and $v, w\in \mathbb{T}$, it follows that $v=w$. Consequently, $\alpha=\tilde{\alpha}$.
\end{proof}
\bigskip
 \begin{proof}[Proof of Theorem \ref{th1}]
\emph{ Necessity.} Conversely, suppose that $\bar{F}(A)\neq\fint_{Q_1}f(x)\,dx$. Using Corollary \ref{th4}, we let $w\in \mathbb{T}$ be a viscosity solution
of $ F(A+D^2w,x)=f(x)-\fint_{Q_1}f(x)\,dx+\bar{F}(A)$.  Clearly, $v, w\in W^{2,p}$ for $p>n$. From the strong maximum principle  for strong solutions and $v, w\in \mathbb{T}$, we obtain  $v=w$, which yields  $\bar{F}(A)=\fint_{Q_1}f(x)\,dx$, a contradiction.

\emph{ Sufficiency.} It is clear from Corollary 1.2.
 \end{proof}
\bigskip
At the end of this section, we   collect some properties of the homogenisation operator $\bar{F}$, which come from \cite{Ev92}.
 \bigskip
 \begin{lem}\label{le8}
 (i) If the  operator $F$ is uniformly elliptic, so is the homogenisation operator $\bar{F}$ with the same ellipticity constants as the operator $F$.

 (ii) If the operator $F$ is concave in $\mathcal{S}^{n\times n}$, so is the homogenisation operator $\bar{F}$ in $\mathcal{S}^{n\times n}$.
 \end{lem}
 \begin{proof}
 (i) To obtain a contradiction, suppose that there exist some $M, N\in \mathcal{S}^{n\times n}$ and $N\geq 0$ such that
 \begin{equation}\label{eq9}
\bar{F}(M+N)-\bar{F}(M) <\lambda \|N\|.
\end{equation}
Let $v^{M}$, $v^{M+N}\in \mathbb{T}$ be  viscosity solutions of
 \begin{equation}\label{eq11}
   \left\{
\begin{aligned}
   &F(D^2v^{M}(x)+M,x)=f(x)-\fint_{Q_1}f(x)\,dx+\bar{F}(M),\\
&F(D^2v^{M+N}(x)+M+N,x)=f(x)-\fint_{Q_1}f(x)\,dx+\bar{F}(M+N).
\end{aligned}
\right.
 \end{equation}
We now claim that
\begin{equation}\label{eq10}
  F(D^2v^{M+N}(x)+M,x) <F(D^2v^{M}(x)+M,x)
\end{equation}
in the viscosity sense.
To see this, let $\phi\in C^2(\mathbb{R}^n)$ and $v^{M+N}-\phi$ has a local minimum at a point $x_0\in \mathbb{R}^n$. In view of (\ref{eq9}), (\ref{eq11}) and the uniform ellipticity of $F$, we have
\begin{equation*}
  \begin{split}
    F(D^2\phi(x_0)+M,x_0) & \leq F(D^2\phi(x_0)+M+N,x_0)-\lambda \|N\| \\
      & \leq f(x_0) -\fint_{Q_1}f(x)\,dx+\bar{F}(M+N)-\lambda \|N\|\\
      &<f(x_0)-\fint_{Q_1}f(x)\,dx+\bar{F}(M)\\
      &=F(D^2v^{M}(x_0)+M,x_0),
  \end{split}
\end{equation*}
which establishes (\ref{eq10}).
Owing to (\ref{eq10}) and the  strong maximum principle for viscosity solutions, we discover
$$v^{M}-v^{M+N}=c$$
for some constant $c$, a contradiction to (\ref{eq10}). It follows that
$$\lambda \|N\| \leq \bar{F}(M+N)-\bar{F}(M)$$
for any $M, N\in \mathcal{S}^{n\times n}$ and $N\geq 0$.
The same argument works for
$$ \bar{F}(M+N)-\bar{F}(M)\leq \Lambda \|N\|.$$

(ii) For later contradiction, let us suppose that there exist some $M, N\in \mathcal{S}^{n\times n}$ such that
$$\bar{F}\left(\frac{M+N}{2}\right)<\frac{1}{2}\bar{F}(M)+\frac{1}{2}\bar{F}(N).$$
 Let $v^{M}$, $v^{N}$, $v^{\frac{M+N}{2}}\in \mathbb{T}$ be  viscosity solutions of
 \begin{equation*}
   \left\{
\begin{aligned}
   &F(D^2v^{M}(x)+M,x)=f(x)-\fint_{Q_1}f(x)\,dx+\bar{F}(M),\\
&F(D^2v^{N}(x)+N,x)=f(x)-\fint_{Q_1}f(x)\,dx+\bar{F}(N),\\
&F\left(D^2v^{\frac{M+N}{2}}(x)+\frac{M+N}{2},x\right)=f(x)-\fint_{Q_1}f(x)\,dx+\bar{F}\left(\frac{M+N}{2}\right).
\end{aligned}
\right.
 \end{equation*}
 For $\varepsilon>0$, let $v_{\varepsilon}^{M}$ be the mollification of $v^{M}$ in $\mathbb{R}^n$ and $v_{\varepsilon}^{N}$ be the mollification of $v^{N}$ in $\mathbb{R}^n$. Then we have,  in view of concavity of $F$,

 $$F\left(\frac{D^2v_{\varepsilon}^{M}(x)+D^2v_{\varepsilon}^{N}(x)}{2}+\frac{M+N}{2},x\right) \geq \frac{1}{2}F\left(D^2v_{\varepsilon}^{M}(x)+M,x\right)+\frac{1}{2}F\left(D^2v_{\varepsilon}^{N}(x)+N,x\right).$$
Sending  $\varepsilon$ to zero, we obtain

\begin{equation*}
  \begin{split}
    F\left(\frac{D^2v^{M}(x)+D^2v^{N}(x)}{2}+\frac{M+N}{2},x\right) & \geq \frac{1}{2}F\left(D^2v^{M}(x)+M,x\right)+\frac{1}{2}F\left(D^2v^{N}(x)+N,x\right)\\
      & =\frac{1}{2}\bar{F}(M)+\frac{1}{2}\bar{F}(N)+f(x)-\fint_{Q_1}f(x)\,dx\\
      &>\bar{F}\left(\frac{M+N}{2}\right)+f(x)-\fint_{Q_1}f(x)\,dx\\
      &=F\left(D^2v^{\frac{M+N}{2}}(x)+\frac{M+N}{2},x\right)
  \end{split}
\end{equation*}
in   the viscosity sense, which together with the strong maximum principle easily yields
$$v^{\frac{M+N}{2}}-\frac{v^{M}+v^{N}}{2}=c$$
for some constant $c$. But we obtain a contradiction.
 \end{proof}
 \bigskip
\begin{rmk}
\emph{(i).}  From (i) in Lemma \ref{le8}, it follows that there exists some $t\in \mathbb{R}$ such that $\bar{F}(tI)=\fint_{Q_1}f(x)\,dx$. Hence we  can always find some $A\in \mathcal{S}^{n\times n}$ satisfying  $\bar{F}(A)=\fint_{Q_1}f(x)\,dx$.

\emph{(ii).}
In the property (ii), if $F$ is concave in some open convex  set $\Gamma \subset \mathcal{S}^{n\times n}$, so is $\bar{F}$ in $\Gamma$.
\end{rmk}
\subsection{ Liouville type result}
\noindent

This subsection will be devoted to the proof of Theorem \ref{th2}.  Following the strategy  implemented  by Caffarelli and Li \cite{CL04}, we will divide the proof into a sequence of lemmas.

Throughout  this subsection,  $u$ is a viscosity solution of
$$F(D^2u,x)=f(x)\ \ \ \mbox{in} \ \ \ \mathbb{R}^n$$
with quadratic growth (\ref{eq36}), where $F$  satisfy (H1), (H2) and (H3) and $f\in C(\mathbb{R}^n)$ is periodic. The oscillation $\beta(x,x_0)$ satisfies the condition (\ref{eq46}).

 For $R\geq 1$, let
 $$u_R(x)=\frac{u(R x)}{R^2},\ \ \ x\in \mathbb{R}^n.$$
  \bigskip
 \begin{lem}\label{le5}
There exists some  $A\in \mathcal{S}^{n\times n}$ such that a subsequence
$\{u_{R_i}\}_{i=1}^\infty$ converges uniformly in the $C^1$ norm  to $Q(x):=\frac{1}{2}x^TAx$ in any compact set of $\mathbb{R}^n$
with $$\bar{F}(A)=\fint_{Q_1}f(x)\,dx.$$
\end{lem}
\begin{proof}
\emph{Step 1: }
We will show that $$u_{R_i} \rightarrow Q\  \ \mbox{in}\  \ C^{1}_{loc}(\mathbb{R}^n)\ \ \ \mbox{as} \ \ \ R_i \rightarrow \infty.$$

 We set $\tilde{f}(x)=f(Rx)$ and $\tilde{F}(M,x)=F(M,Rx)$ for all $x\in \mathbb{R}^n$ and $M\in \mathcal{S}^{n\times n}$. Then we consider the function
 \begin{equation*}
  \begin{split}
   \tilde{\beta}_R(x,x_0)=& \sup_{M\in \mathcal{S}^{n\times n}\setminus \{0\}}\frac{|\tilde{F}(M,x)-\tilde{F}(M,x_0)|}{\|M\|}\\
      =& \sup_{M\in \mathcal{S}^{n\times n}\setminus \{0\}}\frac{|F(M,Rx)-F(M,Rx_0)|}{\|M\|}.
  \end{split}
\end{equation*}
Now we check that
 $$
 \left(\fint_{Q_r(x_0)} |\tilde{\beta}_R(x,x_0)|^n\,dx\right)^{\frac{1}{n}}=
   \left(\fint_{Q_{Rr}(Rx_0)} |\beta(x,Rx_0))|^n\,dx\right)^{\frac{1}{n}}\leq \beta_0, $$
   which follows from the periodicity of $\beta$ and (\ref{eq46}). Therefore, since $u_R$ is a viscosity solution of
   $$\tilde{F}(D^2u_R(x),x)=\tilde{f}(x),\ \ \ x\in Q_{2r} \ (r>1),$$
we have the  $W^{2,p}$ estimate $(p>n)$ (see \cite[Theorem 7.1]{CC95}). Furthermore, by the aid of the Sobolev imbedding theorem, we obtain the $C^{1,\alpha}$ estimate $(0<\alpha<1-\frac{n}{p})$, that is,
\begin{equation*}
  \begin{split}
   \|u_R\|_{L^{\infty}(Q_r)} +r\|Du_R\|_{L^{\infty}(Q_r)}+r^{1+\alpha}[Du_R]_{\alpha; Q_r}& \leq C (\|u_R\|_{L^{\infty}(Q_{2r})}+4r^2\|\tilde{f}\|_{L^{\infty}(Q_{2r})})\\
      & \leq C (1+4r^2+4r^2\|f\|_{L^{\infty}(\mathbb{R}^n)})
  \end{split}
\end{equation*}
for some constant $C>0$, 
 which implies that we  extract a subsequence $\{u_{R_i}\}_{i=1}^{\infty}$  such that
$$u_{R_i} \rightarrow Q\ \ \ \mbox{in}\ \ \ C^{1}_{loc}(\mathbb{R}^n)$$
with $$|Q(x)|\leq C|x|^2,\ \ \ x\in \mathbb{R}^n$$
for some constant $C>0$.

\emph{Step 2: $Q(x)=\frac{1}{2}x^{T}Ax.$} We claim that $Q$ is a viscosity
solution of
\begin{equation}\label{eq42}
 \bar{F}(D^2Q)=\fint_{Q_1}f(x)\,dx
\end{equation}
in $\mathbb{R}^n$.
We first prove that $Q$ is a  viscosity subsolution of (\ref{eq42}).
 Fix $\phi\in C^2(\mathbb{R}^n)$ and suppose $Q-\phi$ has a strict local maximum at $x_0$ with $Q(x_0)=\phi(x_0)$. We intend to prove
 $$\bar{F}(D^2\phi(x_0))-\fint_{Q_1}f(x)\,dx\geq 0.$$
 Suppose, to the contrary, that
 $$\delta:=\bar{F}(D^2\phi(x_0))-\fint_{Q_1}f(x)\,dx<0.$$
 Applying Theorem \ref{th4} to $A=D^2\phi(x_0)$, we let $v\in \mathbb{T}$ be a viscosity solution of
 \begin{equation}\label{eq37}
   F(D^2\phi(x_0)+D^2v,x)=f(x)+\bar{F}(D^2\phi(x_0))-\fint_{Q_1}f(x)\,dx.
 \end{equation}
 Introduce the perturbed test function
 $$\phi^{R_i}(x)=\phi(x)+\frac{1}{R_i^2}v\left(R_i x\right),\ \ \ x\in \mathbb{R}^n.$$

 We claim that
 $$F\left(D^2\phi^{R_i}(x),R_ix\right)-f\left(R_i x\right)\leq \frac{\delta}{2},\ \ \ x\in B_r(x_0)$$
 in the viscosity sense for some sufficiently small $r>0$. To see this, we fix $\psi\in C^{\infty}(\mathbb{R}^n)$ such that $\phi^{R_i}-\psi$ has a minimum at a point $x_1\in B_r(x_0)$ with
 $$\phi^{R_i}(x_1)=\psi(x_1).$$
 Then it is obvious that
$$\eta(y):=v(y)-R_i^2\left(\psi\left(\frac{1}{R_i} y\right)-\phi\left(\frac{1}{R_i} y\right)\right)$$
has a minimum at $y_1=R_ix_1$. Furthermore, observing that $v$ is a viscosity solution of (\ref{eq37}), we have
$$F(D^2\phi(x_0)+D^2\psi(x_1)-D^2\phi(x_1),R_ix_1)-f\left(R_ix_1\right)\leq \delta,$$
which implies that
$$F(D^2\psi(x_1),R_ix_1)-f\left(R_ix_1\right)\leq \frac{\delta}{2}$$
for small enough $r>0$. The claim is proved.

In view of $\phi^{R_i}, u_{R_i}\in W^{2,p}$ $(p>n)$ and
\begin{equation*}
   \left\{
\begin{aligned}
   &F(D^2\phi^{R_i}(x),R_ix)-f\left(R_ix\right)\leq \frac{\delta}{2},\ \ \ x\in B_r(x_0),\\
&F(D^2u_{R_i}(x),R_ix)-f\left(R_ix\right)=0, \ \ \ x\in B_r(x_0),
\end{aligned}
\right.
 \end{equation*}
 the comparison principle for strong solutions leads to
 $$(u_{R_i}-\phi^{R_i})(x_0)\leq \max_{\partial B_r(x_0)}(u_{R_i}-\phi^{R_i}).$$
 In addition, letting $R_i \rightarrow \infty$, we obtain
 $$(Q-\phi)(x_0)\leq \max_{\partial B_r(x_0)}(Q-\phi).$$
But, since $Q-\phi$ has a strict local maximum at $x_0$,  we obtain a contradiction.
In the same manner, we can show that $Q$ is a  viscosity supersolution of (\ref{eq42}).

From the Evans-Krylov theorem and the properties in Lemma \ref{le8}, it follows that there exists some $A\in \mathcal{S}^{n\times n}$ such that
$$Q(x)=\frac{1}{2}x^{T}Ax.$$
\end{proof}
\bigskip
We recall  that
$$E=\{k_1 e_1+\cdots +k_n e_n: k_1,\ldots, k_n\ \ \mbox{are integers}\},$$
and the second order difference quotients
$$\Delta_{e}^2u(x)=\frac{u(x+e)+u(x-e)-2u(x)}{\|e\|^2}\ \ \mbox{for}\ \ e\in E.$$
\bigskip
\begin{lem}\label{le3}
For all $e\in E$, we have
$$F_{ij}(D^2u(x),x)D_{ij}(u(x+e)+u(x-e)-2u(x))\geq 0, \ \ \ \mbox{a.e.}\ \ \  x\in \mathbb{R}^n,$$
where $F_{ij}(D^2u(x),x)$  is the subdifferential of $F$ at $D^2u(x)$.
\end{lem}
\begin{proof}
Due to \cite[Theorem 7.1]{CC95},  $u\in W^{2,p}$ for $p>n$.
By the concavity of $F$ in $M$ and the periodicity of $F$ in $x$, we have
\begin{equation*}
  \begin{split}
    F(D^2u(x+e),x) & \leq F(D^2u(x),x)+F_{ij}(D^2u(x),x)D_{ij}(u(x+e)-u(x)), \\
      F(D^2u(x-e),x) & \leq F(D^2u(x),x)+F_{ij}(D^2u(x),x)D_{ij}(u(x-e)-u(x)).
  \end{split}
\end{equation*}
The result follows immediately from the periodicity of $f$.
\end{proof}
\bigskip
\begin{lem}\label{le4}
For all $e\in E$, we have
$$\sup_{x\in \mathbb{R}^n}\Delta_{e}^2u(x)<\infty.$$
\end{lem}
\begin{proof}
By the $W^{2,p}$ estimate $(p>n)$ (see \cite[Theorem 7.1]{CC95}), we obtain
\begin{equation}\label{eq25}
  r^{2-\frac{n}{p}}\|D^2u\|_{L^p(B_r)}\leq C (\|u\|_{L^{\infty}(B_{2r})}+r^{2-\frac{n}{p}}\|f\|_{L^p(B_{2r})})
\end{equation}
 for some constant $C>0$ depending only on $n$, $p$, $\lambda$ and $\Lambda$.

 Since we can write
 $\Delta_{e}^2u(x)=\int_{-1}^1\frac{e^{T}}{\|e\|} D^2u(x+te)\frac{e}{\|e\|}(1-|t|)\,dt$, then the quadratic growth (\ref{eq36}) and (\ref{eq25}) give
 \begin{equation*}
   \begin{split}
     \int_{B_r}|\Delta_{e}^2u(x)|^p\,dx & = \int_{B_r}\left|\int_{-1}^1\frac{e^{T}}{\|e\|} D^2u(x+te)\frac{e}{\|e\|}(1-|t|)\,dt\right|^p\,dx\\
       & \leq 2^{p-1}\int_{B_r} \int_{-1}^1\left|\frac{e^{T}}{\|e\|} D^2u(x+te)\frac{e}{\|e\|}(1-|t|)\right|^p\,dt\,dx\\
       &=2^{p-1}\int_{-1}^1  \int_{B_r} \left|\frac{e^{T}}{\|e\|} D^2u(x+te)\frac{e}{\|e\|}(1-|t|)\right|^p\,dx\,dt\\
       &\leq 2^{p-1} \int_{-1}^1  \int_{B_{r+\|e\|}} \left\| D^2u(x)\right\|^p\,dx\,dt\\
       &\leq C(r+\|e\|)^n(1+\|f\|_{L^{\infty}(\mathbb{R}^n)})^p.
   \end{split}
 \end{equation*}
 By Lemma \ref{le3} and the local maximum principle, this gives rise to a pointwise estimate
 \begin{equation*}
\begin{split}
 \sup_{x\in B_r} \Delta_{e}^2u(x)&\leq  C\left(\frac{1}{|B_{2r}|}\int_{B_{2r}}\left|\Delta_{e}^2u(x)\right|^p\,dx \right)^{\frac{1}{p}}  \\
    & \leq C\left(1+\frac{\|e\|}{r}\right)^\frac{n}{p}(1+\|f\|_{L^{\infty}(\mathbb{R}^n)})
\end{split}
\end{equation*}
for some constant $C>0$.
\end{proof}
\bigskip
We will adopt  Caffarelli and Li's \cite{CL04} arguments to carry out the rest proof.
\bigskip
\begin{lem}\label{le6}
For all $e\in E$, we have
$$\sup_{x\in \mathbb{R}^n}\Delta_{e}^2u(x)=\frac{e^{T}Ae}{\|e\|^2}.$$
\end{lem}
\bigskip
To proceed, we choose $b\in \mathbb{R}^n$ such that
$$w(e_k)=w(-e_k), \ k=1,\ldots,n, $$
where $$w(x):=u(x)-\frac{1}{2}x^{T}Ax-b\cdot x$$
satisfies $F(A+D^2w,x)=f$ in $\mathbb{R}^n$.
Since $\bar{F}(A)=\fint_{Q_1}f(x)\,dx$, there exists some $v\in \mathbb{T}$ satisfying
$$F(A+D^2v,x)=f$$
in the  viscosity sense. Next, using the $W^{2,p}$ $(p>n)$ regularity, we conclude that  $h:=w-v$ is of class $W^{2,p}$ and satisfies a linear elliptic equation
 $$a_{ij}(x)D_{ij}h(x)=0,\ \ \ \mbox{a.e.} \ \ x\in \mathbb{R}^n,$$
 where $a_{ij}$ satisfies the elliptic condition with  ellipticity  constants $0<\lambda \leq \Lambda < \infty$.
Consequently, to prove that $h$ is constant, by the harnack inequality, it remains to show that $h$ is bounded from above.
\bigskip
\begin{lem}\label{le7}
$$\sup_{\mathbb{R}^n}h< + \infty.$$
\end{lem}
 \bigskip
 Since proofs of the above  two lemmas are the same as  that of  \cite[Proposition 2.3 and Lemma 2.9]{CL04},  the detailed proofs are omitted.

\section{Uniformly elliptic equations on exterior domains}
\noindent

In this section, we establish  the existence and Liouville type results for   quadratic  growth
solutions of uniformly elliptic equations with the periodic data on exterior domains.
 Now we give a proof of Theorem \ref{th10}.
 \bigskip
 \begin{proof}[Proof of Theorem \ref{th10}]
 Without loss of generality, we assume that $\Omega$ contains the origin. Let $r_0=\mbox{daim}\ \Omega+1$.  In view of Theorem \ref{th1}, let $v\in \mathbb{T}$ be a viscosity solution of $F(A+D^2v,x)=f$ in $\mathbb{R}^n$. We set $w:=\frac{1}{2}x^{T}Ax+b\cdot x+v(x)$ and $\bar{C}:=\|w-\varphi\|_{L^{\infty}(\partial \Omega)}$. For $r>r_0$, let $u_r$ be a  viscosity solution of
\begin{equation*}
   \left\{
\begin{aligned}
   &F(D^2u_r,x)=f\ \ \mbox{in}\ \ B_r \backslash \overline{\Omega},\\
&u_r=\varphi\ \  \mbox{on}\ \ \partial \Omega,\\
&u_r=w \ \  \mbox{on}\ \ \partial B_r.
\end{aligned}
\right.
 \end{equation*}
 Clearly, we see that
 \begin{equation*}
   \left\{
\begin{aligned}
   &F(D^2w,x)=f\ \ \mbox{in}\ \ B_r \backslash \overline{\Omega},\\
&w+\bar{C} \geq \varphi\ \  \mbox{on}\ \ \partial \Omega,\\
&w-\bar{C} \leq \varphi\ \  \mbox{on}\ \ \partial \Omega.
\end{aligned}
\right.
 \end{equation*}
 Thus, applying the comparison principle for  strong solutions yields
 $$w-\bar{C}\leq u_r\leq w+\bar{C} \ \ \mbox{in}\ \ B_r \backslash \overline{\Omega}.$$
Hence we can apply the H\"{o}lder estimate  to $F(D^2u_r,x)=f$ in any compact subset $K$ of $\mathbb{R}^n \backslash \overline{\Omega}$ to obtain
$$ \|u_r\|_{C^\alpha(K)}
    \leq  C$$
for some constant $C$ independent of $r$. It follows that there exists a function $u\in C(\mathbb{R}^n \backslash \overline{\Omega})$ and a subsequence of $\{u_r\}_{r=1}^\infty$ that converges uniformly to $u$ in compact sets  of $\mathbb{R}^n \backslash \overline{\Omega}$. Moreover, we have
$$F(D^2u,x)=f\ \ \mbox{in}\ \ \ \mathbb{R}^n \backslash \overline{\Omega}$$
in the viscosity sense with
$$w-\bar{C}\leq u\leq w+\bar{C} \ \ \mbox{in}\ \ \mathbb{R}^n \backslash \overline{\Omega}.$$

We are now in a position to show that $u$ is continuous up to $\partial \Omega$ and coincides with $\varphi$ on $\partial \Omega$. More precisely, for any $x_0\in \partial \Omega$, then we have
$$\lim_{x\in \mathbb{R}^n \backslash \overline{\Omega},\ x\rightarrow x_0}u(x)=\varphi(x_0).$$
To show this, for arbitrary $\epsilon>0$ and fixed $x_0\in \partial \Omega$, by virtue of the continuity of $\varphi$ on $\partial \Omega$, there exists a constant  $0<\delta<1$ such that $|\varphi(x)-\varphi(x_0)|<\epsilon$ if $x \in \partial \Omega\cap B_\delta(x_0)$.
We then define functions $\varphi^{\pm}\in C^{2}(\overline{B_{r_0}\backslash \Omega})$ by
$$\varphi^{\pm}(x)=\varphi(x_0)\pm\left(\epsilon+\frac{2\sup_{\partial \Omega}|\varphi|}{\delta^2}|x-x_0|^2\right).$$
Clearly, we have
\begin{equation*}
   \left\{
\begin{aligned}
&\varphi^{-} \leq \varphi\leq \varphi^{+} \ \  \mbox{on}\ \ \partial \Omega,\\
&\varphi^{-} \leq 0\leq \varphi^{+} \ \  \mbox{on}\ \ \partial B_{r_0}.
\end{aligned}
\right.
 \end{equation*}
 Since the domain $\Omega$  satisfies  a uniform interior sphere condition, we let  $B_{\rho}(z_{x_0})\subset \Omega$ and $\overline{B_{\rho}(z_{x_0})}\cap \partial \Omega={x_0}$.
 Now we consider  functions
 $w^{\pm}\in C^{2}(\overline{B_{r_0}\backslash \Omega})$ defined by
 $$w^{\pm}(x)=\pm \widehat{B}\left(e^{-\widehat{A}\rho^2}-e^{-\widehat{A}|x-z_{x_0}|^2}\right)$$
 for some positive constants $\widehat{A}$ and $\widehat{B}$ to be specified later. For $x\in \overline{B_{r_0}\backslash \Omega}$, choosing large enough $\widehat{A}>\frac{n\Lambda}{2\lambda \rho^2}$ and $\widehat{B}>0$, we obtain
\begin{equation*}
   \begin{split}
       F(D^2w^{+}(x)+D^2\varphi^{+}(x),x)\leq & F(D^2\varphi^{+}(x),x)+\mathcal{M}^{+}(D^2w^{+}(x)) \\
       \leq & F\left(\frac{4\sup_{\partial \Omega}|\varphi|}{\delta^2},x\right)+2e^{-\widehat{A}|x-z_{x_0}|^2}\widehat{B}\widehat{A}(n\Lambda-2\widehat{A}\lambda \rho^2)\\
       \leq & F\left(\frac{4\sup_{\partial \Omega}|\varphi|}{\delta^2},x\right)+2e^{-4\widehat{A}r_0^2}\widehat{B}\widehat{A}(n\Lambda-2\widehat{A}\lambda\rho^2)\\
       \leq & \inf_{\mathbb{R}^n}f,
   \end{split}
 \end{equation*}
 \begin{equation*}
   \begin{split}
      F(D^2w^{-}(x)+D^2\varphi^{-}(x),x)\geq & F(D^2\varphi^{-}(x),x)+\mathcal{M}^{-}(D^2w^{-}(x)) \\
       \geq & F\left(-\frac{4\sup_{\partial \Omega}|\varphi|}{\delta^2},x\right)-2e^{-\widehat{A}|x-z_{x_0}|^2}\widehat{B}\widehat{A}(n\Lambda-2\widehat{A}\lambda\rho^2)\\
       \geq & F\left(-\frac{4\sup_{\partial \Omega}|\varphi|}{\delta^2},x\right)-2e^{-4\widehat{A}r_0^2}\widehat{B}\widehat{A}(n\Lambda-2\widehat{A}\lambda \rho^2)\\
       \geq & \sup_{\mathbb{R}^n}f
   \end{split}
 \end{equation*}
 and
$$\inf_{x\in \partial B_{r_0} }\widehat{B}\left(e^{-\widehat{A}\rho^2}-e^{-\widehat{A}|x-z_{x_0}|^2}\right)\geq \widehat{B}\left(e^{-\widehat{A}\rho^2}-e^{-\widehat{A}(\rho+1)^2}\right)\geq \sup_{\partial B_{r_0}}(|w|+\bar{C}).$$
Clearly, it follows that
$$w^{-}+\varphi^{-}\leq u_r\leq w^{+}+\varphi^{+}\ \  \mbox{on}\ \ \partial \Omega\cup \partial B_{r_0}.$$
Consequently, by the aid of the  comparison principle  for viscosity solutions, we have
 $$w^{-}+\varphi^{-}\leq u_r\leq w^{+}+\varphi^{+}\ \ \mbox{in}\  \ B_{r_0}\backslash \overline{\Omega}.$$
 Furthermore, letting $r\rightarrow \infty$ leads to
$$ w^{-}+\varphi^{-}\leq u\leq w^{+}+\varphi^{+}\ \ \mbox{in}\  \ B_{r_0}\backslash \overline{\Omega},$$
 that is,
 $$ |u(x)-\varphi(x_0)|\leq \epsilon+\frac{2\sup_{\partial \Omega}|\varphi|}{\delta^2}|x-x_0|^2+w^{+}(x),$$
 which immediately implies that $u(x)\rightarrow \varphi(x_0)$ as $x\rightarrow x_0$.

Finally, by the $W^{2, p}$ regularity, $u-w $ is of class $W^{2, p}$ and therefore satisfies a linear elliptic equation. Furthermore, since $u-w$  is bounded in $\mathbb{R}^n \backslash \overline{\Omega}$, we apply the Harnack inequality and the comparison principle to conclude that $\lim_{|x|\rightarrow \infty} (u-w)(x)$ exists.
In particular, if $\frac{\Lambda}{\lambda}< n-1$, we can obtain a more refined error estimate. Indeed, applying the comparison principle to the  Pucci's operators (see \cite[Theorem 1.10]{ASS11} or \cite[Lemma 2.5]{LZP}), we obtain the desired estimate (\ref{eq39}).
 \end{proof}
 \bigskip
 We end up this section with the proof Theorem \ref{th9}.
\bigskip
\begin{proof}[Proof of Theorem \ref{th9}]
For $r>2$, let $u_r$  be a viscosity solution of
\begin{equation*}
   \left\{
\begin{aligned}
   &F(D^2u_r,x)=f\ \ \mbox{in}\ \ B_r,\\
&u_r=u\ \  \mbox{on}\ \ \partial B_r.
\end{aligned}
\right.
 \end{equation*}

 We will show that $u_r$ is  bounded in compact sets  of $\mathbb{R}^n$. For this purpose, we let $\bar{u}\in C^2(\mathbb{R}^n)$ keeping $\bar{u}=u$ outside $B_2$ and
 set $$F(D^2\bar{u},x)=f+g\ \ \ \mbox{in}\ \ \ \mathbb{R}^n,$$
 where $g$ is H\"{o}lder continuous with support in $B_2$.  We  choose $\bar{C}>0$ such that
 $$\bar{C}\frac{1}{4}(\Lambda+\lambda(n-1))\left(1-\frac{\lambda}{\Lambda}(n-1)\right)2^{\frac{1}{2}\left(1-\frac{\lambda}{\Lambda}(n-1)\right)-2}\leq -\|g\|_{L^{\infty}(B_2)}.$$
 Then we consider $$E(x)=\bar{C}|x|^{\frac{1}{2}\left(1-\frac{\lambda}{\Lambda}(n-1)\right)},\ \ \ x\in \mathbb{R}^n\setminus \{0\}.$$
 Since $\bar{u}-u_r$ is bounded in $B_1$, there exists a constant $0<\varepsilon<1$ such that
 $$|\bar{u}(x)-u_r(x)| \leq E(x),\ \ \ x\in \partial B_\varepsilon \cup \partial B_r.$$
 For $x\in B_r\backslash B_\varepsilon$, we obtain
 $$\mathcal{M}^{+}(D^2E(x))\leq g(x)\leq \mathcal{M}^{+}(D^2\bar{u}(x)-D^2u_r(x))$$
 and
 $$\mathcal{M}^{-}(D^2\bar{u}(x)-D^2u_r(x))\leq g(x)\leq \mathcal{M}^{-}(-D^2E(x))$$
 From the comparison principle, it follows that
 $$  \bar{u}(x)-E(x)\leq u_r(x)\leq \bar{u}(x)+E(x),\ \ \ x\in B_r\backslash B_\varepsilon.$$
 Applying the Alexandroff-Bakelman-Pucci estimate to $F(D^2u_r,x)=f$ in $B_2$, we have
 $$\|u_r\|_{L^{\infty}(B_2)}\leq \|u\|_{L^{\infty}(\partial B_2)}+\|E\|_{L^{\infty}(\partial B_2)}+C\|f\|_{L^{\infty}(B_2)},$$
 where $C>0$ depends only on $n$, $\lambda$ and $\Lambda$.
Hence we prove that $u_r$ is  bounded in compact sets  of $\mathbb{R}^n$.
From the H\"{o}lder estimate, it follows that there exists a subsequence of $\{u_r\}_{r=1}^\infty$ that converges uniformly to a function $w\in C(\mathbb{R}^n)$ in compact sets  of $\mathbb{R}^n$.
Moreover, we have
$$F(D^2w,x)=f\ \ \mbox{in}\ \ \ \mathbb{R}^n$$
in the viscosity sense with
\begin{equation}\label{eq45}
  |w(x)-u(x)| \leq E(x),\ \ x\in \mathbb{R}^n \backslash \overline{B_2}
\end{equation}
and therefore
$$|w(x)|\leq C(1+|x|^2), \ \ x\in \mathbb{R}^n$$
for some constant $C>0$. In addition, by Theorem \ref{th2}, we conclude that
$$w(x)=\frac{1}{2}x^{T}Ax+b\cdot x+c+v(x)$$
for  some $A\in \mathcal{S}^{n\times n}$ with $\bar{F}(A)=\fint_{Q_1}f(x)\,dx$, $b\in \mathbb{R}^n$, $c\in \mathbb{R}$ and $v\in \mathbb{T}$.

Finally, due to $F(D^2u,x)=F(D^2w,x)=f$ in $\mathbb{R}^n \backslash \overline{B_1}$ in the viscosity sense and (\ref{eq45}), $u-w$ is a strong solution of a linear elliptic equation and is bounded in $\mathbb{R}^n \backslash \overline{B_1}$.  Following  the arguments in the establishment of the estimate (\ref{eq39}), we obtain
$$|u(x)-w(x)-c^*| \leq C|x|^{1-\frac{\lambda}{\Lambda}(n-1)},\ \ \ x\in \mathbb{R}^n \backslash \overline{B_1}$$
for some $c^*\in \mathbb{R}$ and $C>0$.
\end{proof}

\end{document}